\documentclass[a4paper,abstracton]{scrartcl}

\usepackage[utf8]{inputenc}

\usepackage{amsmath}
\usepackage{amsthm}
\usepackage{amssymb}

\usepackage[usenames,dvipsnames]{color}
\usepackage{graphicx}
\usepackage{subfigure}

\usepackage{color}

\usepackage{amssymb,amsmath,verbatim,tabularx,enumitem,colonequals,graphicx,psfrag,subfigure}

\usepackage{cite}

\newcommand{\cU}{\mathcal{U}}

\newcommand{\X}{\mathcal{X}}

\usepackage{authblk}
\usepackage[normalem]{ulem} 

\usepackage{framed}
\usepackage{rotating}

\newtheorem{theorem}{Theorem}
\newtheorem{lemma}[theorem]{Lemma}

\newtheorem{corollary}[theorem]{Corollary}

\newcommand{\BIGOP}[1]{\mathop{\mathchoice%
{\raise-0.22em\hbox{\huge $#1$}}%
{\raise-0.05em\hbox{\Large $#1$}}{\hbox{\large $#1$}}{#1}}}

\allowdisplaybreaks

\begin{document}

\title{Mixed Uncertainty Sets for Robust Combinatorial Optimization}

\author[1]{Trivikram Dokka\thanks{Email: t.dokka@lancaster.ac.uk}}
\author[2]{Marc Goerigk\thanks{Email: marc.goerigk@uni-siegen.de}}
\author[1]{Rahul Roy\thanks{Corresponding author. Email: r.roy@lancaster.ac.uk}}

\affil[1]{Department of Management Science, Lancaster University, United Kingdom}
\affil[2]{Network and Data Science Management, University of Siegen, Germany}

\date{}

\maketitle

\abstract{In robust optimization, the uncertainty set is used to model all possible outcomes of uncertain parameters. In the classic setting, one assumes that this set is provided by the decision maker based on the data available to her. Only recently it has been recognized that the process of building useful uncertainty sets is in itself a challenging task that requires mathematical support.

In this paper, we propose an approach to go beyond the classic setting, by assuming multiple uncertainty sets to be prepared, each with a weight showing the degree of belief that the set is a ''true'' model of uncertainty. We consider theoretical aspects of this approach and show that it is as easy to model as the classic setting. In an extensive computational study using a shortest path problem based on real-world data, we auto-tune uncertainty sets to the available data, and show that with regard to out-sample performance, the combination of multiple sets can give better results than each set on its own.
}


\section{Introduction}

In this paper we consider combinatorial problems of the form
\[ \min_{\pmb{x}\in\X} \pmb{c} \pmb{x} \tag{P} \]
with $\X\subseteq\{0,1\}^n$ and uncertain cost vector $\pmb{c}$. To find a solution $\pmb{x}$ that still performs well under the possible cost realizations, different approaches have been proposed. These include fuzzy optimization \cite{kasperski2010minmax}, stochastic programming \cite{birge2011introduction}, or robust optimization \cite{goerigk2016algorithm,gabrel2014recent}.

In the robust optimization approach, we assume that all possible cost realizations $\pmb{c}$ are modelled through a so-called uncertainty set $\cU$, and we want to protect against all outcomes without knowledge of a probability distribution. The resulting worst-case problem is then of the form
\[ \min_{\pmb{x}\in\X} \max_{\pmb{c}\in\cU} \pmb{c}\pmb{x} \]
Special cases of this kind have been investigated thoroughly (see \cite{kasperski2016robust} for a recent overview on robust combinatorial optimization). A common assumption in these settings is that the shape of the uncertainty set $\cU$ is known, e.g., one assumes that the uncertainty is interval-based
\[ \cU=\prod_{i\in[n]}[\underline{c}_i,\overline{c}_i], \]
discrete
\[ \cU = \left\{ \pmb{c}^1,\ldots,\pmb{c}^K \right\} \]
or ellipsoidal
\[ \cU = \left\{ \pmb{c}\in\mathbb{R}^n_+ : (\pmb{c}-\hat{\pmb{c}})\pmb{\Sigma}^{-1}(\pmb{c}-\hat{\pmb{c}}) \le \lambda \right\} \]
Note that in the min-max setting, using a discrete uncertainty set is equivalent to using its convex hull $\cU = conv\left(\left\{ \pmb{c}^1,\ldots,\pmb{c}^K \right\}\right)$.
So far, comparatively little research has investigated how a decision maker can actually come up with an uncertainty set that produces a robust solution in accordance with his or her wishes. In the data-driven approach \cite{bertsimas2018data}, we do not assume an uncertainty set to be given, but only data observations, which are usually discrete. These are then used to construct uncertainty sets, e.g, using different approaches from statistical testing.

In \cite{chassein2018variable,chassein2018compromise}, the authors considered a setting in which the shape of the uncertainty set is given, but not its size. Models are introduced by which compromise robust solutions can be found, which perform well on average over all considered uncertainty sizes. Furthermore, in \cite{crespi2018robust} the sensitivity of robust solutions to the uncertainty size was considered.

In the recent paper \cite{CHASSEIN2018}, real-world data modeling traffic speeds in the city of Chicago were used to test different uncertainty sets for shortest path problems experimentally. In particular, discrete uncertainty and ellipsoidal uncertainty sets were found to produce a good trade-off on an out-of-sample evaluation set of scenarios with respect to average performance, worst-case performance and the conditional value at risk (CVaR) criterion.

In this paper, we introduce a novel approach to handle uncertainty in robust optimization. This mixed-uncertainty setting is a direct generalization of the classic robust optimization approach, where we protect against multiple uncertainty sets simultaneously. We demonstrate that this approach is well-suited to data-driven settings, where the decision maker is not able to determine the shape and size of uncertainty a priory. By using \textit{irace} \cite{lopez2011irace}, a software designed to tune algorithms automatically, we determine the best-performing combination of uncertainty sets for the same Chicago test data as used in \cite{CHASSEIN2018}. Our proposed approach is generally applicable, and possible to implement using off-the-shelf software with little theoretical knowledge from the decision maker, making it flexible and attractive for practical purposes.

The rest of this paper is structured as follows. In Section~\ref{sec:mixed}, we formally introduce our setting of mixed uncertainty sets for uncertain combinatorial problems. We demonstrate how to find compact mixed-integer programming models, and point out cases solvable in polynomial time. Section~\ref{sec:exp} presents an extensive computational case study, highlighting the usefulness of our approach for a real-world shortest path problem. We conclude this paper in Section~\ref{sec:conclusions} and point out further research questions.

\section{Mixed Uncertainty Sets}
\label{sec:mixed}

In classic min-max robust combinatorial optimization, we consider the problem
\[ \min_{\pmb{x}\in\X} \max_{\pmb{c}\in\cU} \pmb{c}\pmb{x} \tag{RP}\]
where the uncertainty set $\cU$ contains all possible outcomes of the cost vector $\pmb{c}$. As a direct generalization, it is possible that multiple uncertainty sets $\cU_1,\ldots,\cU_N$ need to be considered. For each $j\in[N]$, we are given a weight $p_j$ that denotes the importance to protect against $\cU_j$, or likelihood of its occurrence.

The resulting weighted robust optimization is then as follows:
\[ \min_{\pmb{x}\in\X} \sum_{j\in[N]} p_j \max_{\pmb{c}\in\cU_j} \pmb{c}\pmb{x} \tag{WRP} \]

There are different settings where the application of a model of the form (WRP) can be useful. For example: (1) We receive different forecasts for future developments of costs, each providing a set of most likely scenarios. Instead of using only one of these sets or merging them, we consider a weighted robust problem, where we assign each forecast an expert estimate $p_j$ whether it can be trusted. (2) We create different uncertainty sets based on different degrees of risk-willingness. We then find a single compromise solution against all levels of risk. For example, one uncertainty set may cover the worst 80\% of outcomes of a multivariate normal distribution, whereas another uncertainty set covers 95\%. We can find a single solution hedging against both worst cases with prescribed weights. (3) We have an original set of observations for our uncertain data, and do not know which shape of uncertainty set may be appropriate for the problem at hand.  Selecting the (right) uncertainty set is itself a decision problem under uncertainty. So far, researchers left this problem to the user/decision-maker and developed approaches for robust optimization after this decision is made offering different choices to this decision. However, this itself can be posed as an optimization or a learning problem. In other words, ideally, we would want to learn the right uncertainty set from data, which may be a mix of many known sets. Also, automatically learning a mixed uncertainty set allows the possibility for the set to be dynamic which can change with a change in the random nature of the underlying uncertain parameters.

Mixing uncertainty sets and tuning the mixing (hyper) parameters to learn the best mix is a challenging task. Research in this direction is in its early stages, the only work to have experimented in this direction is that of \cite{Campbell15_ACC} where authors consider combining ellipsoidal sets within a Bayesian setting for robust linear programs.  Moreover, even under the assumption that we can formulate this learning task and embed it in a holistic algorithmic framework for data-driven robust optimization, there is no guarantee that mixing sets will help in improved robust solutions for any choice of robustness measures. In this spirit, (WRP) is a first step towards automating selection of the right uncertainty set. In this work we would like to investigate if mixing (known) sets can give better robust solutions.

We now consider the complexity of (WRP), and possibilities to formulate it using mixed-integer programming. Note that (WRP) can be written as a compact optimization problem if that is possible for each $\cU_i$, by combing each model. As an example, if $\cU_1$ is a budgeted uncertainty set of the form $\cU_1 = \{ \pmb{c} : c_i = \underline{c}_i + z_i \overline{c}_i, \sum_{i\in[n]} z_i \le \Gamma, z_i\in\{0,1\} \}$ (see \cite{bertsimas2003robust}), then the robust problem (RP) can be written as
\begin{align*}
\min\ &\sum_{i\in[n]} \underline{c}_i x_i + \Gamma\pi + \sum_{i\in[n]} \rho_i \\
\text{s.t. } & \pi + \rho_i \ge  (\overline{c}_i - \underline{c}_i)x_i & \forall i\in[n] \\
& \pmb{x} \in \X
\end{align*}
If $\cU_2$ is a polyhedral uncertainty set of the form $\cU_2 = \{ \pmb{c} : V\pmb{c} \le \pmb{d}, \pmb{c}\ge \pmb{0}\}$, then its robust problem (RP) can be written as
\begin{align*}
\min\ & \pmb{d} \pmb{\alpha} \\
\text{s.t. } & V^t\pmb{\alpha} \ge \pmb{x} \\
& \pmb{\alpha} \ge \pmb{0} \\
& \pmb{x} \in \X
\end{align*}
Combining both uncertainty sets in a weighted robust problem (WRP) using weights $p_1$ for budgeted uncertainty and $p_2$ for polyhedral uncertainty, the combination of models yields
\begin{align*}
\min\ & p_1(\sum_{i\in[n]} \underline{c}_i x_i + \Gamma\pi + \sum_{i\in[n]} \rho_i) + p_2 \pmb{d} \pmb{\alpha} \\
\text{s.t. } & \pi + \rho_i \ge  (\overline{c}_i - \underline{c}_i)x_i & \forall i\in[n] \\
& V^t\pmb{\alpha} \ge \pmb{x} \\
& \pmb{x} \in \X
\end{align*}

As (WRP) is an generalization of (RP), it is at least as hard. Hence, cases of (WRP) that involve uncertainty sets $\cU_i$ for which we already know that the classic robust optimization problem is NP-hard, will be NP-hard as well.

We therefore focus on cases where (RP) is still polynomially solvable. First, notice that when $\cU_1,\ldots,\cU_N$ are interval uncertainty sets of the form $\cU_j = \prod_{i\in[n]} [\underline{c}^j_i,\overline{c}^j_i]$, then
\[ \sum_{j\in[N]} p_j \max_{\pmb{c}\in\cU_j} \pmb{c}\pmb{x} = \sum_{i\in[n]} (\sum_{j\in[N]} p_j\overline{c}^j_i)x_i \]
which is a problem of nominal type. Hence, (WRP) with interval sets has the same complexity as the nominal problem (P).

Now let $\cU_1,\ldots,\cU_N$ be budgeted uncertainty sets with $\cU_j = \{ \pmb{c} : c_i = \underline{c}^j_i + z_i \overline{c}^j_i, \sum_{i\in[n]} z_i \le \Gamma^j, \pmb{z}\in\{0,1\}^n \}$. We can write (WRP) as
\begin{align*}
\min\ & \sum_{j\in[N]} p_j(\sum_{i\in[n]} \underline{c}^j_i x_i + \Gamma^j\pi^j + \sum_{i\in[n]} \rho^j_i) \tag{B-WRP}\\
\text{s.t. } & \pi^j + \rho^j_i \ge  (\overline{c}^j_i - \underline{c}^j_i)x_i & \forall i\in[n],j \in[N] \\
& \pmb{x} \in \X
\end{align*}
\begin{lemma}\label{wrpbudgeted}
There is an optimal solution to (B-WRP), where for all $j\in[N]$, we have $\pi^j = \overline{c}^j_{i(j)} - \underline{c}^j_{i(j)}$ for some $i(j)\in[n]$, or $\pi^j = 0$.
\end{lemma}
\begin{proof}
For some fixed $\pmb{x}$, the remaining optimization problem in $\pi^j$ and $\pmb{\rho}^j$ can be decomposed in to $N$ independent sub-problems. For each sub-problem, the result is known from the classic budgeted uncertainty case (see \cite{bertsimas2003robust}).
\end{proof}

\begin{theorem}
For constant $N$, (B-WRP) can be solved in polynomial time, if (P) can be solved in polynomial time.
\end{theorem}
\begin{proof}
We rewrite (B-WRP) as
\begin{align*}
\min\ &\sum_{j\in[N]} p_j \left( \sum_{i\in[n]} \underline{c}^j_ix_i + \Gamma^j\pi^j + \sum_{i\in[n]} \left[(\overline{c}^j_i-\underline{c}^j_i)x_i-\pi^j\right]_+ \right) \\
\text{s.t. } & \pmb{x}\in\X
\end{align*}
where $[y]_+ = \max\{y,0\}$. Note that
\[  \left[(\overline{c}^j_i-\underline{c}^j_i)x_i-\pi^j\right]_+ = \left[\overline{c}^j_i-\underline{c}^j_i-\pi^j\right]_+ x_i\]
Using Lemma~\ref{wrpbudgeted}, we enumerate all $(n+1)^N$ options for the $\pi^j$ variables. For fixed values of $\pmb{\pi}$, (B-WRP) reduces to
\[ \min_{\pmb{x}\in\X} \sum_{i\in[n]} \left(\sum_{j\in[N]} p_j \left(\underline{c}^j_i + \left[\overline{c}^j_i-\underline{c}^j_i-\pi^j\right]_+ \right) \right)  x_i + \sum_{j\in[N]} p_j\Gamma^j\pi^j,\]
which is of the nominal type.
\end{proof}

Note that this approach cannot be used if $N$ is unbounded (i.e., part of the input). For this case, we can still identify cases that are tractable in the following.

For $\pmb{x}\in\X$, denote by $X = \{i\in[n]: x_i=1\}$ the corresponding set of chosen items. We rewrite the objective of (WRP) as
\[ f(X) = \sum_{j\in[N]} p_j \left( \sum_{i\in X} \underline{c}^j_i + wc^j(X) \right) \]
where $wc_j(X)$ denotes the sum of the $\Gamma^j$ largest values $d^j_i:=\overline{c}^j_i-\underline{c}^j_i$ for $i\in X$, i.e., 
 \[ wc^j(X) = \max \left\{ \sum_{i\in X} d^j_i z_i : \sum_{i\in[n]} z_i \le \Gamma^j, \pmb{z}\in\{0,1\}^n \right\} \]

\begin{theorem}\label{submod}
Function $f$ is submodular.
\end{theorem}
\begin{proof}
We first show that $wc^j$ is a submodular function. As the sum of submodular functions is still submodular, it follows that $f$ is submodular.

To see that $wc^j$ is submodular, let $S,T\subseteq [n]$ be given. Let $\ell(X)$ be the up to $\Gamma^j$ items that maximize $\sum_{i\in \ell(X)} d^j_i$. Note that $\ell(S\cup T) \subseteq \ell(S) \cup \ell(T)$, while $\ell(S\cap T)$ is not necessarily in $\ell(S) \cup \ell(T)$. Figure~\ref{submod} illustrates these relationships.

\begin{figure}[htbp]
\begin{center}
\includegraphics[width=.4\textwidth]{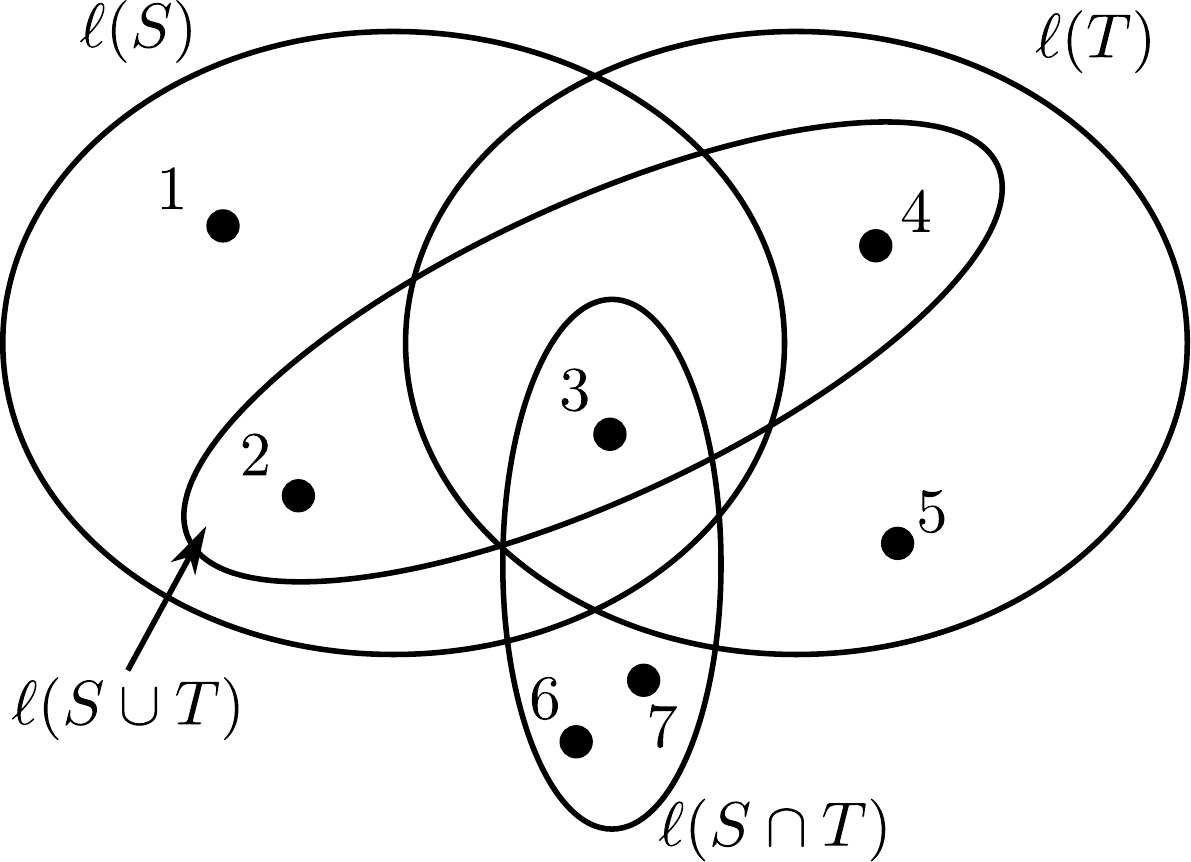}
\caption{Illustration for the proof of Theorem~\ref{submod}.}
\end{center}
\end{figure}

For any element $i\in \ell(S\cap T)$, one of two cases hold: If $i\in \ell(S\cup T)$, then also $i\in\ell(S)$ and $i\in\ell(T)$. If $i\notin \ell(S\cup T)$, then there is an element $k\in(\ell(S)\cup\ell(T))\setminus\ell(S\cup T)$ with $d^j_k \ge d^j_i$. Hence,
\[ wc^j(S) + wc^j(T) \ge wc^j(S\cup T) + wc^j(S\cap T) \]
and the claim follows.
\end{proof}
 
\begin{corollary}\label{submodcor}
For problems where $\X$ is a matroid, (B-WRP) can be solved in polynomial time, even for unbounded $N$.
\end{corollary}

Corollary~\ref{submodcor} applies to, e.g., the weighted robust spanning tree problem, or the weighted robust selection problem.
 
Note that these results also apply to combinations of interval and budgeted uncertainty sets.

Finally, we consider the approximability of problems with discrete uncertainty sets. It is a well-known result (see, e.g., \cite{aissi2009min,chassein2018scenario}) that for classic robust combinatorial optimization problems with discrete uncertainty $\cU=\{\pmb{c}^1,\ldots,\pmb{c}^K\}$, an optimal solution to the midpoint scenario $\frac{1}{K}\sum_{k\in[K]} \pmb{c}^k$ gives a $K$-approximation. We extend this result to mixed discrete uncertainty sets.

Let $\cU^j=\{\pmb{c}^{1,j},\ldots,\pmb{c}^{K_j,j}\}$ for all $j\in[N]$. We construct the mixed-uncertainty midpoint scenario
\[ \hat{\pmb{c}} = \sum_{j\in[N]} p_j\left(\frac{1}{K_j}\sum_{k\in[K_j]} \pmb{c}^{k,j} \right)\]
and set $K^{max} = \max_{j\in[N]} K_j$.
\begin{theorem}
An optimal solution to problem (P) with costs $\hat{\pmb{c}}$ is a $K^{max}$-approximation to (WRP) with mixed discrete scenarios.
\end{theorem}
\begin{proof}
Let $\hat{\pmb{x}}$ be an optimal solution to $\hat{\pmb{c}}$, and let $\pmb{x}^*$ be an optimal solution to (WRP). Then
\begin{align*}
\sum_{j\in[N]}p_j\max_{k\in[K_j]} \pmb{c}^{k,j}\hat{\pmb{x}} & \le \sum_{j\in[N]} p_j \left(\sum_{k\in[K_j]} \pmb{c}^{k,j}\hat{\pmb{x}}\right) \\
&\le K^{max} \sum_{j\in[N]} p_j \left(\frac{1}{K_j} \sum_{k\in[K_j]} \pmb{c}^{k,j}\hat{\pmb{x}}\right) \\
&\le K^{max} \sum_{j\in[N]} p_j \left(\frac{1}{K_j} \sum_{k\in[K_j]} \pmb{c}^{k,j}\pmb{x}^*\right) \\
&\le K^{max} \sum_{j\in[N]} p_j \max_{k\in[K_j]} \pmb{c}^{k,j}\pmb{x}^* \\
\end{align*}
Hence, the objective of solution $\hat{\pmb{x}}$ in problem (WRP) is at most $K^{max}$ times the objective of an optimal solution.
\end{proof}

\section{Tuning Experiments}
\label{sec:exp}
Algorithms which solve optimization problems involve a number of parameters, and these parameters can be carefully tuned so that the performance of the algorithms is optimized. For this purpose, a tuning tool to automatically configure optimization algorithms called \textit{irace} \cite{lopez2011irace} has been developed.
We use \textit{irace} to find the best-performing combination of uncertainty sets for (WRP)
a given set of instances of the robust shortest path problem.

\subsection{Experimental Setup}
We used the same real-world shortest path data as first introduced in \cite{CHASSEIN2018}. The graph models the City of Chicago, and consists of 538 nodes and 1308 arcs. The \textit{morning} data set was used to model the problem uncertainty, containing 271 scenarios that represent morning rush hours during week days. Each scenario contains the travel speed for each arc.

In our experiments, three uncertainty sets, ellipsoidal, convex hull and interval, are used in generating the mixed uncertainty sets. Additionally, we set $N=3$, meaning that we combine up to three uncertainty sets; however, the tuning algorithm can also choose less. We refer to these three sets as the \textit{parent} uncertainty sets. The choice of parent uncertainty sets is driven by computational efficiency and relative performance of uncertainty sets demonstrated in \cite{CHASSEIN2018}. It is possible that two parent sets are of the same type, e.g., a combination that uses two convex hull uncertainty sets and one ellipsoidal uncertainty set.
Corresponding to each parent set, we have a scaling parameter $\lambda$, which takes up values in pre-defined intervals (including the interval limits). The intervals are so chosen to reflect a reasonable range of choices for a decision maker.
For ellipsoidal uncertainty we use $\lambda \in [0, 20]$, while we use $\lambda \in [0, 1]$ for interval and convex hull uncertainty.
Parameter $\lambda$ represents the size of the uncertainty set, in relation to the observed scenarios. For a formal parameter definition, we refer to \cite{CHASSEIN2018}.

Moreover, we also associate a weight $p_j$ to each parent set which corresponds to the significance of the parent in the mixed uncertainty set and, therefore, in the objective function. 

In our experiments, our objective is to find a path that is robust when driving during morning rush hours, modeled through 271 scenarios. We use 75\% of the scenarios (in-sample data) to construct the uncertainty sets, and we evaluate solutions in-sample and out-sample separately. To this end, we generate 600 random \textit{s-t} pairs over the node set, which fulfill a minimum distance criterion to avoid nodes that are close to each other.
To find a balanced evaluation of all methods, we use three performance criteria: (1) The average objective function value over all \textit {s-t} pairs and all scenarios, denoted as Avg.
(2) The average of the worst-case objective function value for each \textit {s-t} pair, denoted as Max.
(3) The average value of the worst 5\% of objective values for each \textit {s-t} pair (as in the conditional value at risk measure), denoted as CVaR. 
 We see that each evaluation is comprised of three performance measures, i.e., each solution is assigned a three-dimensional objective vector. 
 
 For \textit{irace} to tune the algorithmic parameters of mixed uncertainty sets, the algorithm must return the cost function as a single value. Hence, three additional parameters were defined, with each parameter being a weight to each performance measure in the objective vector. The cost function then becomes a weighted sum of the performance measures. However, it is important to note that these three weights are user-defined and not automatically configured by \textit{irace}. The three weights corresponding to their performance measures are sampled from the interval $[0, 1]$ with a step size of $0.1$, and only those values are retained that sum up to one (which makes a total of 66 such combinations of weights). Each combination results in an algorithm, and therefore, a solution for the shortest path problem using the mixed uncertainty sets. Hence, in total, for mixed uncertainty sets, we have nine automatically configurable parameters and three fixed user-defined parameters. We used a fixed computational budget of 10,000 experiments for a given \textit{irace} run. 
 
As a comparison, we generate 41 possible values for the scaling parameter $\lambda$ when each parent uncertainty set is separately used to compute solutions to the robust shortest path problem; this does not involve any parameter tuning as solutions are computed for different sizes of the parent uncertainty sets. The 41 values are equidistant, i.e., for ellipsoidal sets we use a step size of $0.5$ for $\lambda$, whereas for interval and convex hull we use a step size of $0.025$. We use the 41 models obtained for each parent set to compare the performance of mixed set solutions with the performance of the solutions delivered by the pure parent sets.

\subsection{Results}
We present the performance of the uncertainty sets, both mixed and parent, in Figures~\ref{fig:Max-vs-Avg} and \ref{fig:CVaR-vs-Avg}. 
As mentioned earlier, a total of 41 objective space vectors are obtained for each parent uncertainty set, where as a total of 66 objective space vectors are obtained for the mixed uncertainty sets. Each objective space vector among the 66 vectors for mixed uncertainty sets corresponds to a unique best parametric configuration obtained for different combinations of the weights of the performance measures. Moreover, each element (performance measure) of the objective space vector corresponding to a unique configuration is an average of all the values of that performance measure taken over all the scenarios and \textit{s-t} pairs. This also holds true for the solutions obtained using the parent uncertainty sets. Figure~\ref{fig:Max-vs-Avg} shows the trade-off curves between the Avg and the Max performance measures for both in-sample and out-sample data, while Figure~\ref{fig:CVaR-vs-Avg} shows the trade-off curves between the Avg and the CVaR performance measures for both in-sample and out-sample data. All the values are in unit of minutes; a lower value indicates better performance. Hence, good trade-off solutions should move from top left to the bottom right of the curves. Some values for interval sets are not visible, as they are outside the plotted ranges.

Figure~\ref{fig:in-max-avg} shows that for in-sample performance, convex hull solutions dominate all other solutions for the Avg vs Max trade-offs by construction; however, the mixed set solutions closely match the convex hull solutions. Interval solutions perform the worst among all, especially for higher values of the scaling parameter. Figure~\ref{fig:in-cvar-avg} shows that for in-sample performance, ellipsoidal solutions exhibit the best Avg vs CVaR trade-offs among all the solutions but are closely matched by the mixed set solutions. However, convex hull solutions perform worse than both ellipsoidal and mixed set, and interval solutions closely match the performance of the convex hull solutions. Summarily, both ellipsoidal and convex hull solutions do not exhibit stability across both the trade-off curves, i.e., while convex hull dominates the in-sample Avg vs Max curves, ellipsoidal exhibits best performance for in-sample Avg vs CVaR trade-offs. Mixed uncertainty set solutions exhibit stability across both in-sample Avg vs Max and in-sample Avg vs CVaR trade-offs as they closely match the respective performance of both convex hull and ellipsoidal solutions.  

\begin{figure}[htb]
\begin{center}
\subfigure[In-Sample\label{fig:in-max-avg}]{\includegraphics[width=0.45\linewidth]{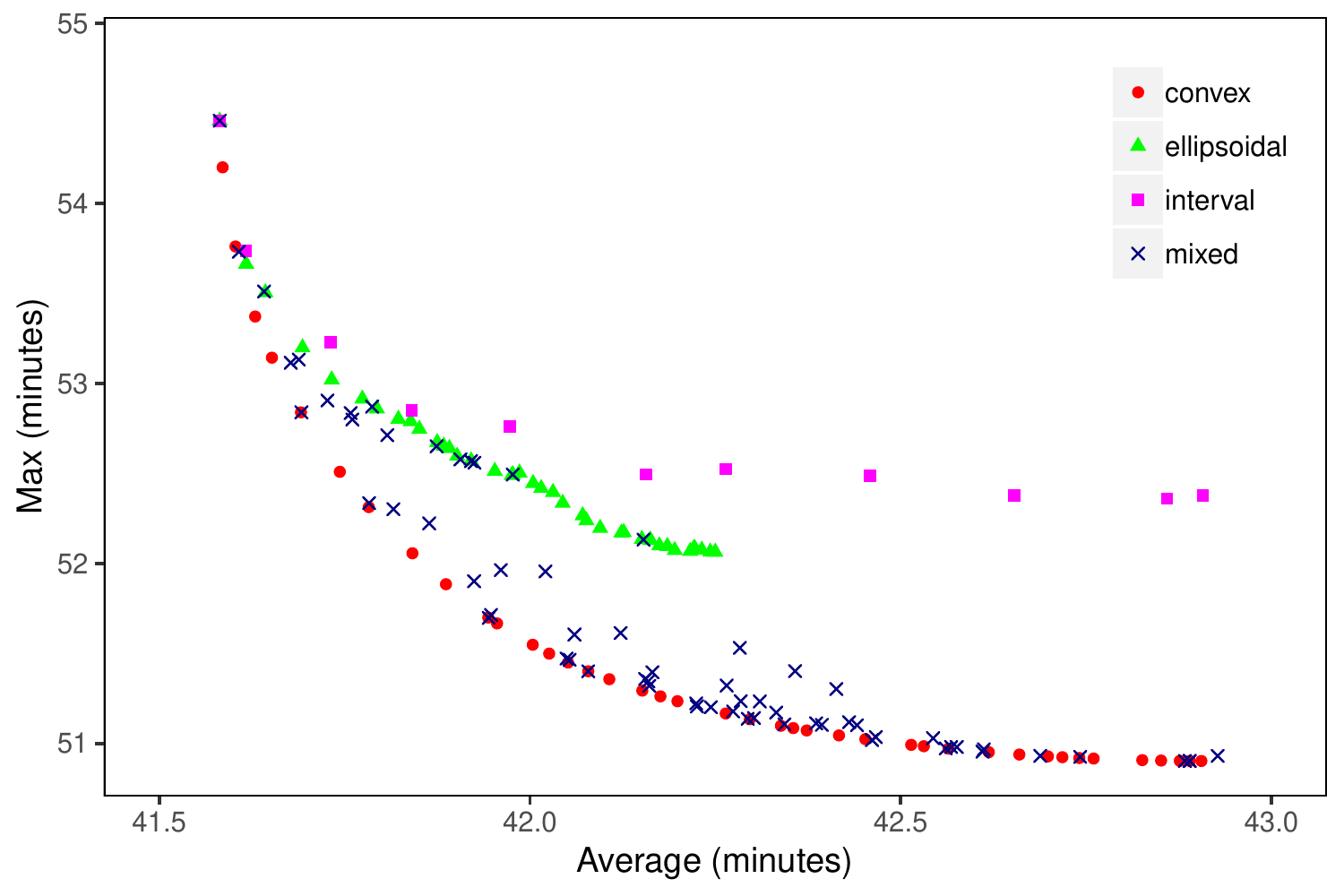}}
\subfigure[Out-Sample\label{fig:out-max-avg}]{\includegraphics[width=0.45\linewidth]{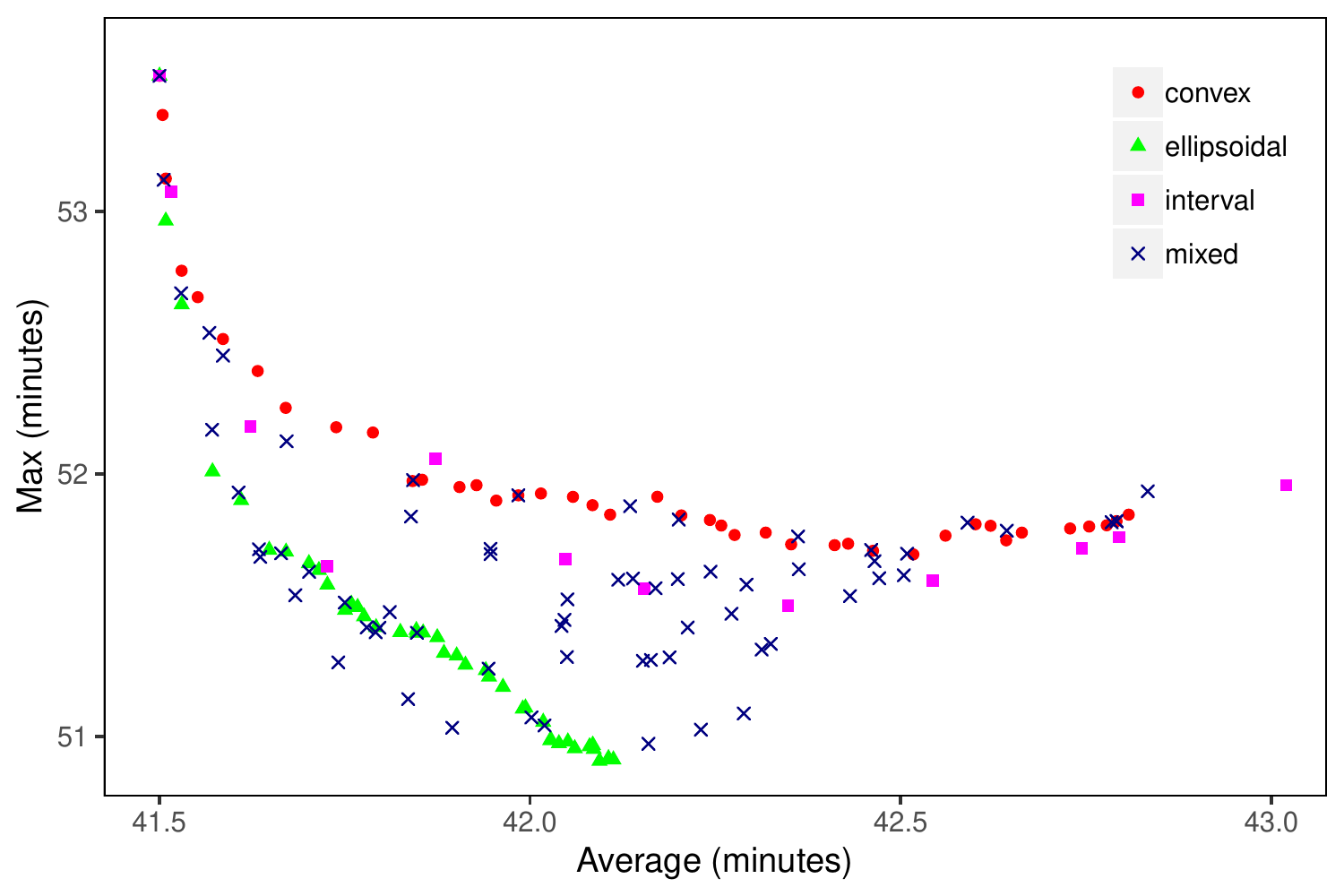}}
\caption{Average vs Worst-Case Performance}
\label{fig:Max-vs-Avg}
\end{center}
\end{figure}

\begin{figure}[htb]
\begin{center}
\subfigure[In-Sample\label{fig:in-cvar-avg}]{\includegraphics[width=0.45\linewidth]{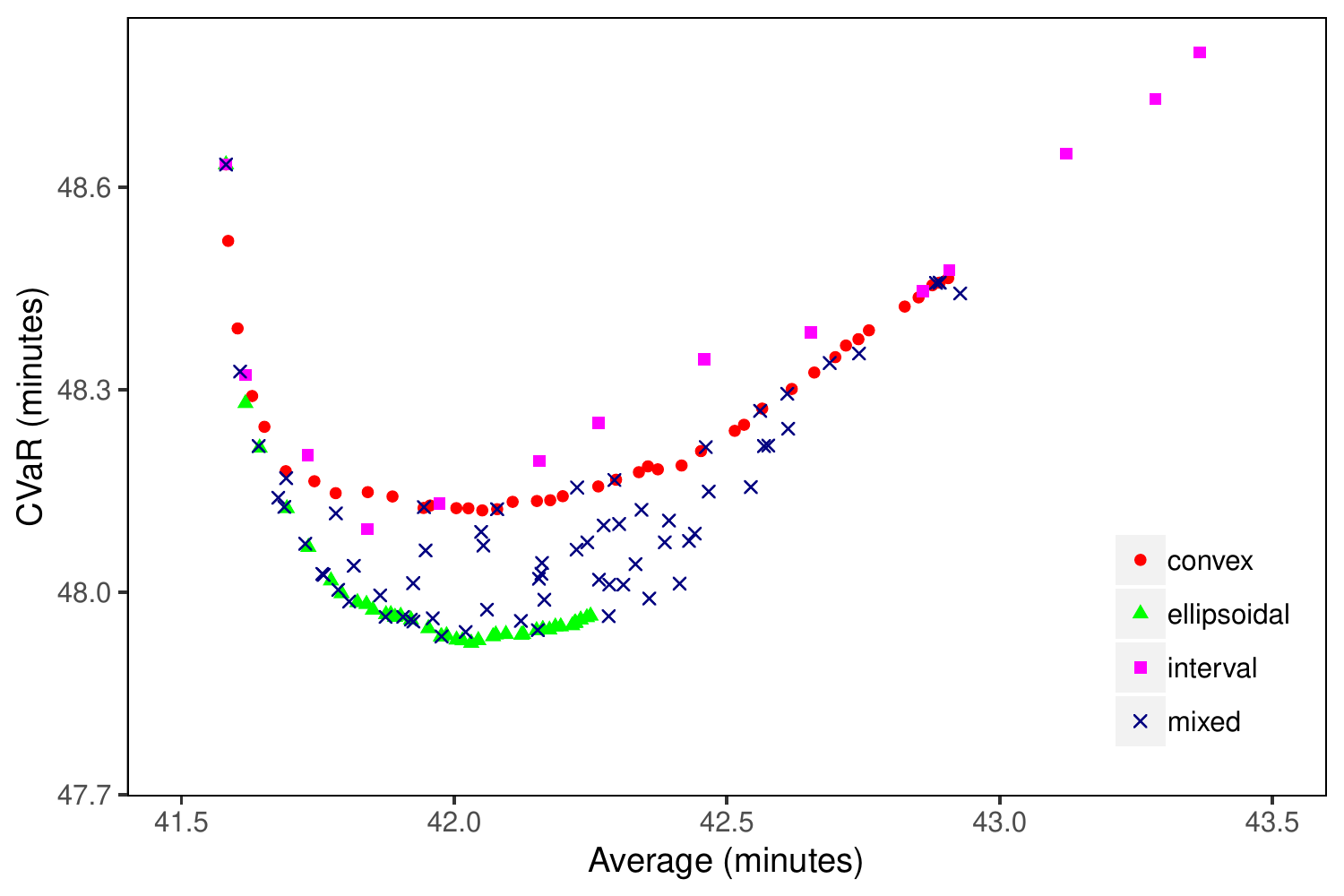}}
\subfigure[Out-Sample\label{fig:out-cvar-avg}]{\includegraphics[width=0.45\linewidth]{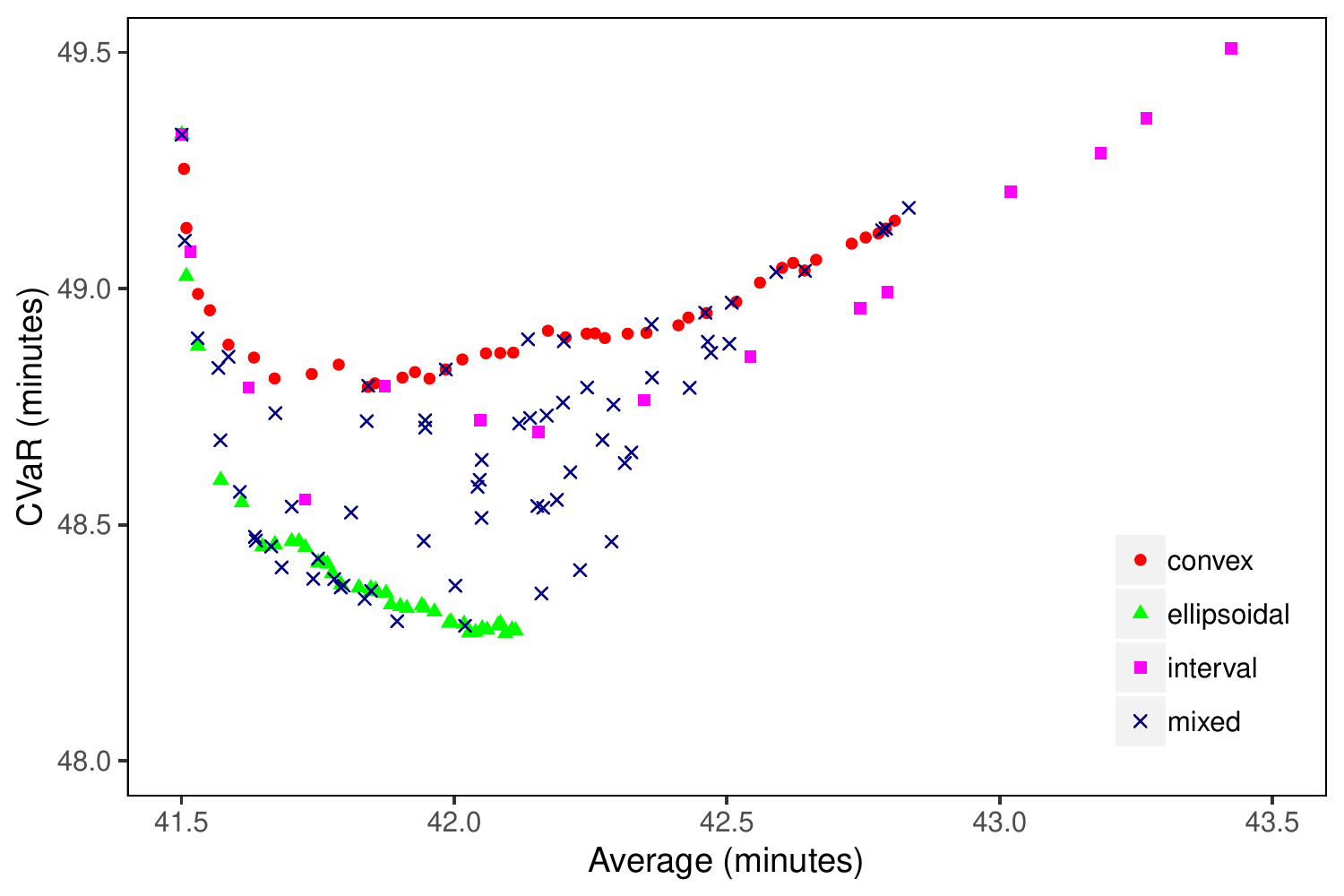}}
\caption{Average vs CVaR Performance}
\label{fig:CVaR-vs-Avg}
\end{center}
\end{figure}

When out-sample performance is considered for both Avg vs Max and Avg vs CVaR trade-offs (Figures~\ref{fig:out-max-avg} and \ref{fig:out-cvar-avg}), mixed set solutions can dominate all other solutions. Convex hull solutions perform even worse than interval solutions for both the trade-offs, implying that they are over-fitted to the data. The performance of ellipsoidal solutions improve from in-sample to out-sample compared to the performance of convex hull and interval solutions for Avg vs Max trade-offs, but lose their best trade-off performance for Avg vs CVaR only to mixed set solutions.

The key aspect to note is that, even though mixed solutions only use the three parent sets, their combination can outperform each separately, i.e., we can observe a synergy effect when mixing uncertainty sets. 

While the performance comparison among all the uncertainty sets help us establish that mixed uncertainty sets not only exhibit stability over both in-sample and out-sample data but also perform better than the parent sets for certain configurations, we do not observe how each uncertainty set performs for each \textit{s-t} pair separately, as we explained earlier that each solution delivered by an uncertainty set is an average of all the values taken over all the scenarios and \textit{s-t} pairs. As an instance, to understand how each uncertainty set performs for each \textit{s-t} pair, we choose a representative configuration for each uncertainty set to compare the performance. The description of the configurations which are compared can be found below.

\begin{itemize}
    \item Mixed Uncertainty Set: We use a mix of convex hull and ellipsoidal uncertainty. The scaling parameter, $\lambda$, and the weight parameter, $p_j$, of convex hull are 0.2234 and 0.7502 respectively. For the ellipsoidal set, the scaling and weight parameters take the values 5.4609 and 0.9796 respectively.
    \item Convex Hull: The scaling parameter is 0.2
    \item Ellipsoidal: The scaling parameter is 6.0
    \item Interval: The scaling parameter is 0.075
\end{itemize}

Figures~\ref{fig:Boxplot-Avg}, \ref{fig:Boxplot-Max}, \ref{fig:Boxplot-CVaR} show the comparison in performance for each measure (Avg, Max and CVaR), both in-sample and out-sample, for every pair of mixed set and parent set, i.e., we compare performance for mixed and convex hull sets, for mixed and ellipsoidal sets and for mixed and interval sets. This is achieved by taking the difference in the solutions delivered by the mixed set and the comparing parent set for each \textit{s-t} pair; all the values are in unit of minutes. While a negative value indicates that the solution delivered by the mixed set is better than the parent set, a positive value indicates the opposite. For the sake of clarity in the plots, we filter out the sample containing only zero values as they indicate that both the sets deliver exactly the same solution. In our case, for each pair of mixed and parent set, the number of \textit{s-t} pairs for which the solutions differed are as follows: 131 (21.8\%) for mixed and convex hull sets; 91 (15.2\%) for mixed and ellipsoidal sets; and 129 (21.5\%) for mixed and interval sets.

\begin{figure}[htb]
\begin{center}
\subfigure[In-Sample\label{fig:bplot-iavg}]{\includegraphics[width=0.45\linewidth]{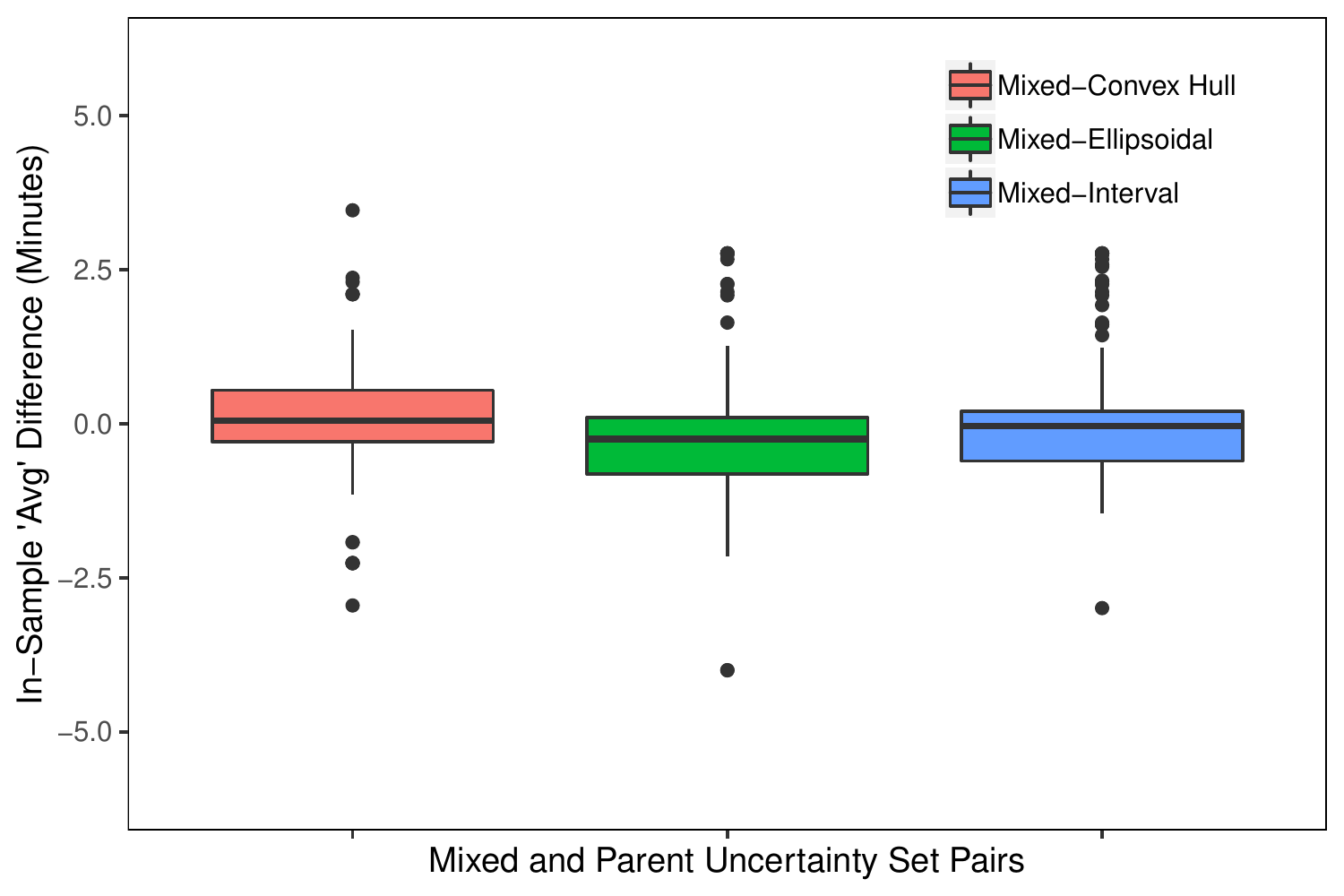}}
\subfigure[Out-Sample\label{fig:bplot-oavg}]{\includegraphics[width=0.45\linewidth]{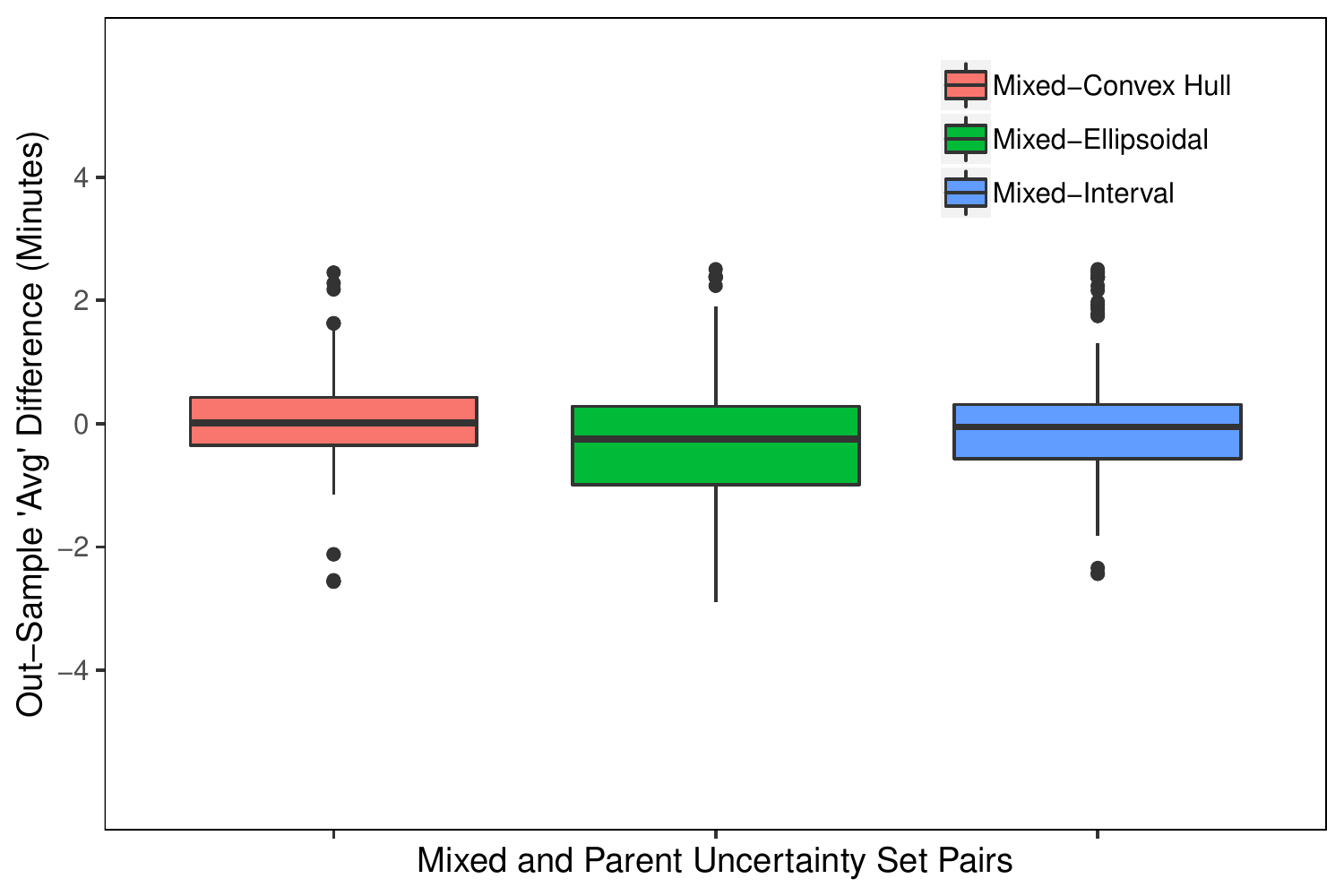}}
\caption{Difference in Average (Avg) Performance Measures}
\label{fig:Boxplot-Avg}
\end{center}
\end{figure}

\begin{figure}[htb]
\begin{center}
\subfigure[In-Sample\label{fig:bplot-imax}]{\includegraphics[width=0.45\linewidth]{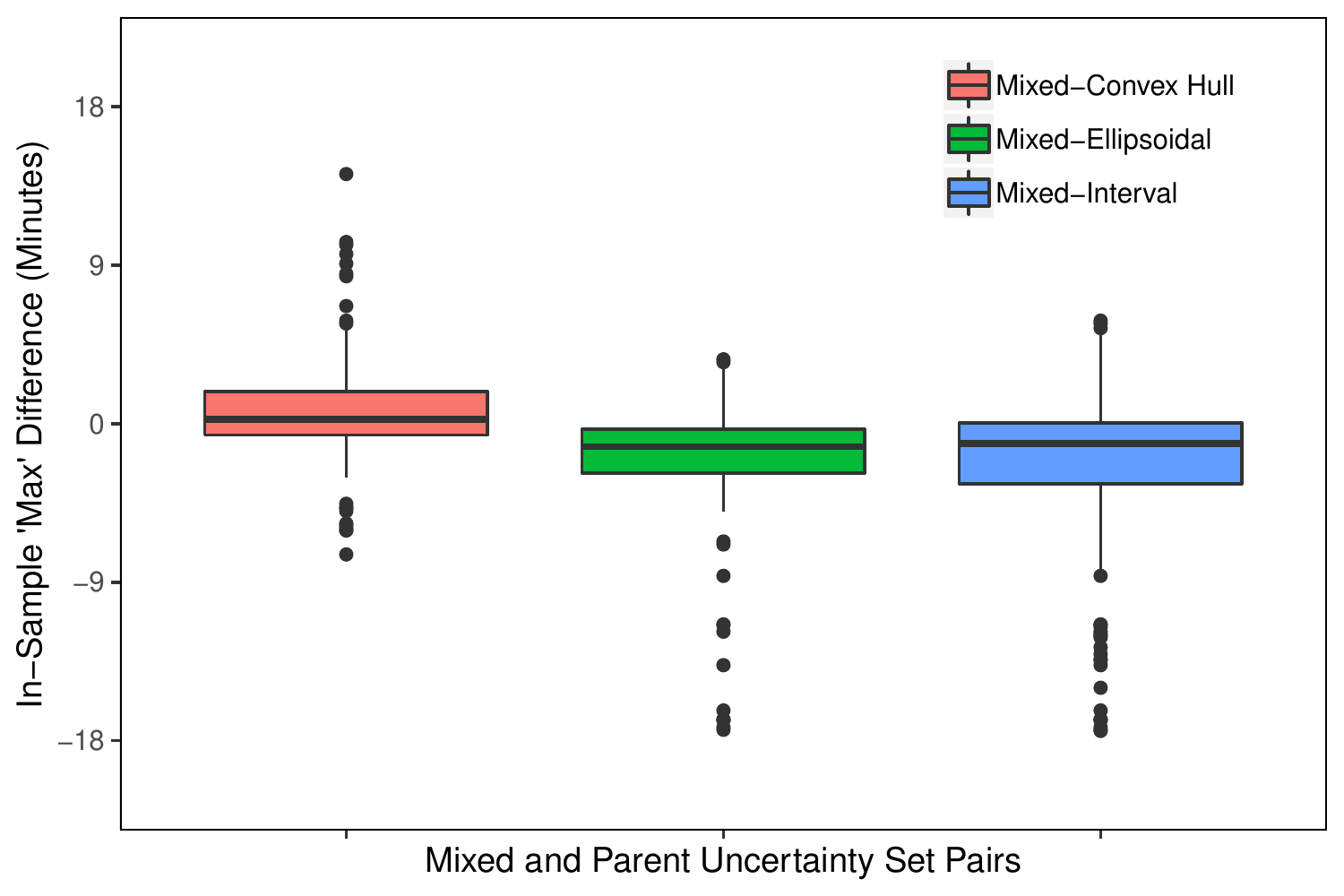}}
\subfigure[Out-Sample\label{fig:bplot-omax}]{\includegraphics[width=0.45\linewidth]{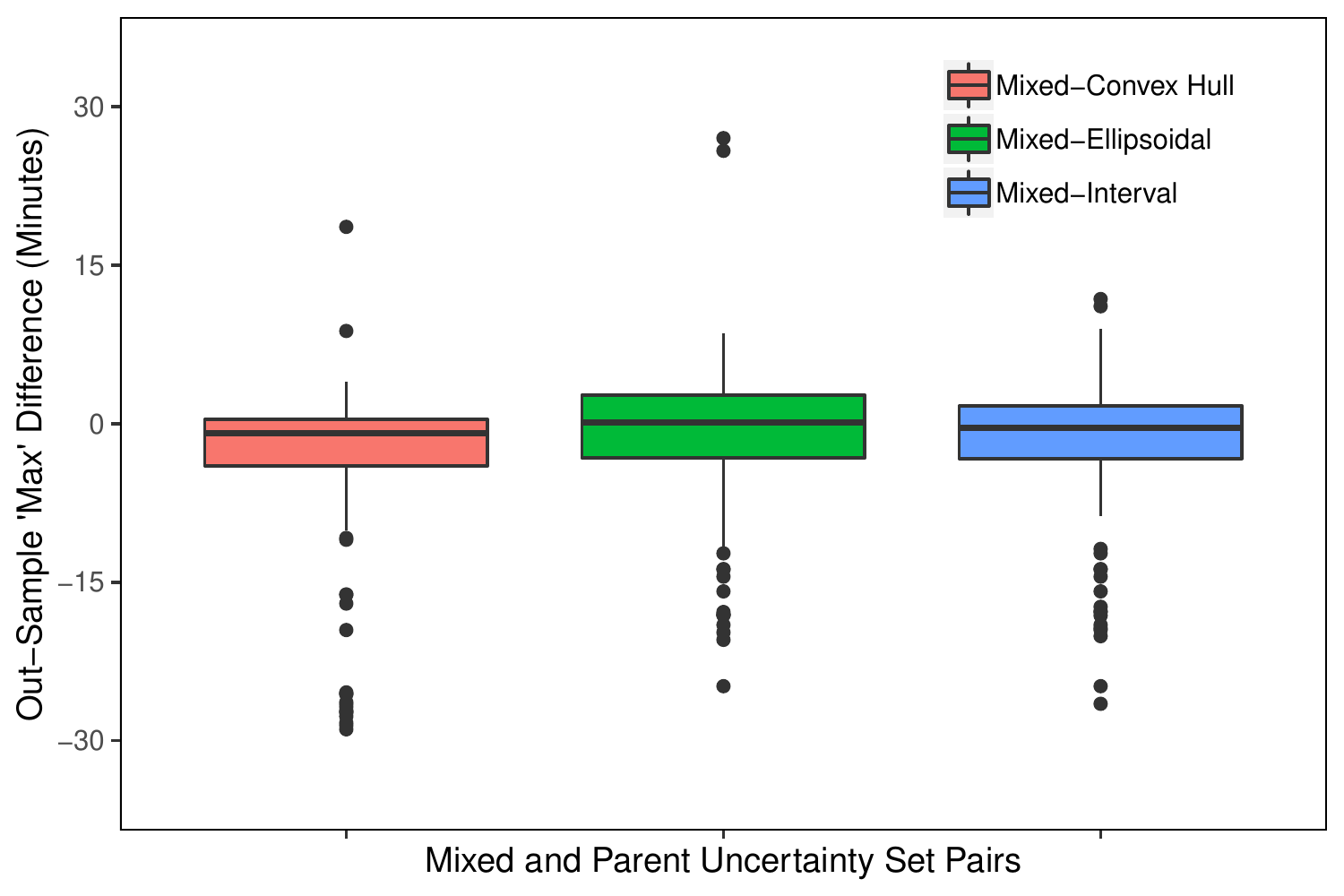}}
\caption{Difference in Worst-Case (Max) Performance Measures}
\label{fig:Boxplot-Max}
\end{center}
\end{figure}

\begin{figure}[htb]
\begin{center}
\subfigure[In-Sample\label{fig:bplot-icvar}]{\includegraphics[width=0.45\linewidth]{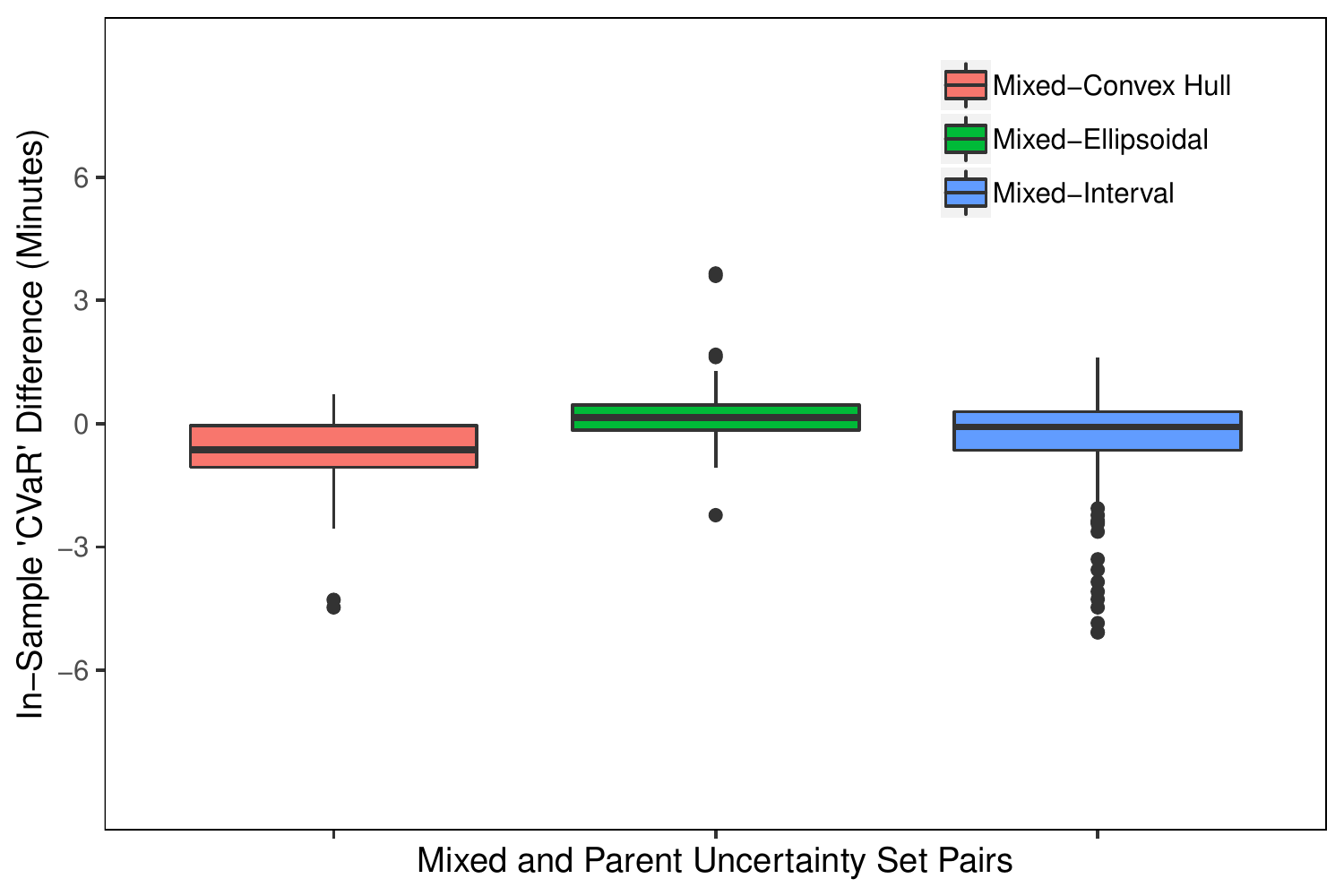}}
\subfigure[Out-Sample\label{fig:bplot-ocvar}]{\includegraphics[width=0.45\linewidth]{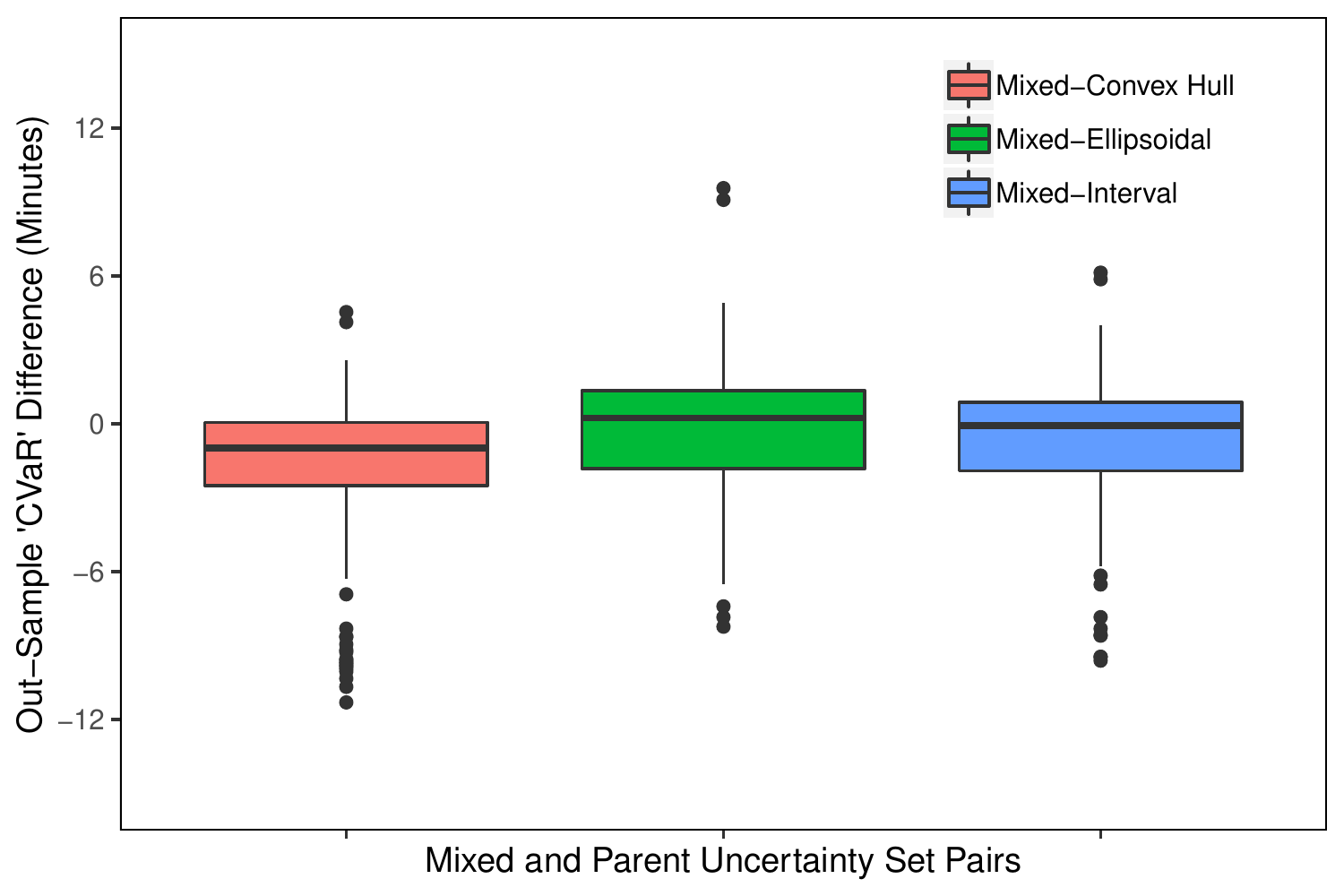}}
\caption{Difference in Average of Worst 5\% (CVaR) Performance Measures}
\label{fig:Boxplot-CVaR}
\end{center}
\end{figure}

Figure~\ref{fig:Boxplot-Avg} compares the Avg performance measure (both in-sample and out-sample) for each pair of mixed and parent set. We see that at least 50\% of the in-sample solutions given by convex hull is better than that given by the mixed uncertainty set. For ellipsoidal-mixed set and interval-mixed set pairs, at least 50\% of the in-sample solutions given by mixed set is better than those of the parent sets. Moreover, for in-sample performance, the difference in the values of the performance measures (margin) by which convex hull solutions are better is higher than the margin by which mixed set solutions are better. For the other two pairs of mixed and parent sets, mixed set performs better by higher margins than by the margin when it performs worse. For out-sample performance, mixed set solutions still dominate the solutions of ellipsoidal and interval, in both number and value. However, convex hull solutions although dominate for at least 50\% of the solutions, the mixed set solutions dominate in margin of performance, i.e. when mixed set solutions are better for out-sample data, they are better by a higher margin.

Similar analysis for the Max performance measure from Figure~\ref{fig:Boxplot-Max} leads us to conclude that while at least 50\% of the in-sample solutions given by convex hull is not only better than those given by mixed set but also better by higher margin, for out-sample data, mixed set solutions dominate convex hull solutions, in both number and margin. For ellipsoidal-mixed set and interval-mixed set pairs, in-sample solutions given by mixed set is better than those of the parent sets, in both number and margin. For out-sample performance, mixed set solutions still dominate the solutions of interval set, in both aspects, but now ellipsoidal solutions are found to be slightly dominating the mixed set solutions, in both aspects.
Figure~\ref{fig:Boxplot-CVaR} compares the CVaR performance measure (both in-sample and out-sample) for each pair of mixed and parent set. We see that mixed set solutions dominate the convex hull and interval solutions in both aspects of number and margin across both in-sample and out-sample data. However, ellipsoidal solutions are found to be clearly dominant over the mixed set solutions, in both number and margin, across both in-sample and out-sample data. 

We summarize our findings from our experiment on the robust shortest path problem with real-world data.

Convex hull solutions show good in-sample performance for Avg vs Max trade-offs, but are not stable when facing out-sample scenarios. The in-sample performance of convex hull solutions is closely matched by the mixed set solutions. Convex hull solutions perform worse for Avg vs CVaR trade-offs across both in-sample and out-sample scenarios. 

Interval solutions do not perform well in general, but are easy to find and can be a reasonable approach for smaller values of the scaling parameter.

Ellipsoidal solutions exhibit stability over both in-sample and out-sample performance for both Avg vs Max and Avg vs CVaR trade-offs and offer a good approach over a wide range of the scaling parameter. In addition, they also deliver the best CVaR performance across both in-sample and out-sample data among all the uncertainty sets. However, solutions given by ellipsoidal set are dominated by the mixed set solutions in the region of best trade-offs for both in-sample and out-sample scenarios when averaged over all the \textit{s-t} pairs and scenarios. 

Mixed set solutions closely match the in-sample performance of convex hull and ellipsoidal solutions for Avg vs Max and Avg vs CVaR trade-offs respectively, but they dominate all the other solutions when facing out-sample scenarios, especially in the region of best trade-offs, i.e., mixed set solutions are found to deliver the best solutions among all the uncertainty sets for certain configurations. Besides, mixed set solutions exhibit stability across both in-sample and out-sample scenarios. 

Mixed set solutions do not always deliver the best solution compared to the parent sets for each \textit{s-t} pair when we consider the best trade-offs among the performance measures; but when the mixed sets give better solutions than the parent sets, the margin by which they are better is much higher than the margin by which they are worse when they under-perform compared to the parent sets. This makes the average value of the solutions given by mixed sets over all the \textit{s-t} pairs and scenarios better than the parent sets, and hence, make them a better option to find better trade-offs among the objective vector elements.

\section{Conclusion}
\label{sec:conclusions}

In this paper we have proposed a mixed uncertainty set approach to robust combinatorial optimization. Our results give a strong evidence in support of mixed sets giving superior solutions compared to individual approaches. This evidence paves way to further investigation into developing efficient algorithmic framework for building such mixed sets. For example, as an immediate future extension of this work, one option is to consider a Bayesian type approach, instead of a black-box optimizer like \textit{irace}, which updates the weight on each parent set under consideration with every new bit of data collected.

\end{document}